\documentclass[preprint,10pt]{elsarticle}

\usepackage{mathtools}
\usepackage{amsmath}
\usepackage{graphicx}                          
\usepackage{amssymb}
\usepackage{lineno,hyperref}
\usepackage{amsthm}

	\newtheorem{lemma}{Lemma}
	
\journal{TBD}
\begin{document}
\begin{frontmatter}


\title{Integer patterns in Collatz sequences}
\author{Zenon B. Batang}
\address{Coastal and Marine Resources Core Laboratory, King Abdullah University of Science and Technology, Thuwal 23955, Saudi Arabia}
\ead{zenon.batang@kaust.edu.sa}

\begin{abstract}
The Collatz conjecture asserts that repeatedly iterating $f(x) = (3x + 1)/2^{a(x)}$, where $a(x)$ is the highest exponent for which $2^{a(x)}$ exactly divides $3x+1$, always lead to $1$ for any odd positive integer $x$. Here, we present an arborescence graph constructed from iterations of $g(x) = (2^{e(x)}x - 1)/3$, which is the inverse of $f(x)$ and where $x \not \equiv [0]_3$ and $e(x)$ is any positive integer satisfying $2^{e(x)}x - 1 \equiv [0]_3$, with $[0]_3$ denoting $0\pmod{3}$. The integer patterns inferred from the resulting arborescence provide new insights into proving the validity of the conjecture.
\end{abstract}

\begin{keyword}
Collatz conjecture \sep Arborescence \sep Covering system  \sep Cycle \sep Proof
\MSC[2010] 11B99
\end{keyword}

\end{frontmatter}

\section{Introduction} \label{sec1}
Denote by $\mathbb{N}=\{1,2,3,\ldots\}$ the set of natural numbers and let $\mathbb{N}_0 \coloneqq \mathbb{N} \cup \{0\}$, $\mathbb{E} \coloneqq 2\mathbb{N}_0$ and $\mathbb{U}\coloneqq 2\mathbb{N}_0+1$. Write $[r]_q$ for $r\pmod{q}$, where $r$ is \textit{least nonnegative residue} modulo $q$. Every $x \in \mathbb{U}$ can be stated as $x = 3\mu_r+ r$, where $\mu_r \in \mathbb{N}_0$ is the \textit{multiple} of $x \equiv [r]_3$. Note that $\mu_0, \mu_2 \in \mathbb{U}$ and $\mu_1 \in \mathbb{E}$. Define the mapping $f \colon \mathbb{U} \to \mathbb{U}$ by
\begin{equation} \label{eq1}
	f(x) = (3x + 1)/2^{a(x)},
\end{equation}
where $a(x) \in \mathbb{N}$ is the highest exponent for which $2^{a(x)}$ exactly divides $3x+1$. The famous \textit{Collatz conjecture} asserts that for every $x \in \mathbb{U}$, there exists $k \in \mathbb{N}_0$ such that $f^{k}(x)=1$, where $f^k$ denotes $k$ compositions of $f$. For an initial input $x_0$, let $f^{k}(x_0) = x_k$ and any $k$ iterations of $f$ on $x_0$ generate a sequence of odd integers, termed \textit{Collatz sequence} or \textit{trajectory}, denoted by
\begin{equation} \label{eq2}
	S^{k}(x_0) = \{x_0, x_1, \ldots, x_k\}_{k \in \mathbb{N}_0}.
\end{equation}
We have $f^{0}(x_0)=x_0$ as the identity map and $f^{\infty}(1)=1$ forms the \textit{trivial cycle} (\textit{loop}). Collatz conjecture claims that \eqref{eq1} always yields the trivial cycle for any $x \in \mathbb{U}$, i.e. repeatedly iterating \eqref{eq1} always yields a \textit{convergent} Collatz sequence.

The Collatz conjecture is an intriguing problem in mathematics that has remained unsolved for over $80$ years despite its apparent simplicity. While probabilistic heuristics, stochastic models, and computational verifications suggest that the conjecture is likely true, it has so far resisted any attempts at a complete proof by different mathematical approaches. One can refer to \cite{Laga1}, \cite{Laga2} and \cite{Cham} for an overview of the conjecture, with an annotated bibliography in two parts provided by Lagarias (\cite{Laga3}, \cite{Laga4}). Recognizing the notorious difficulty of the Collatz conjecture, the prolific mathematician Paul Erd{\H o}s famously stated that \textit{``Mathematics is not yet ready for such problems''} (\cite{Laga1}). In this paper, we present new insights into integer patterns that underlie the truth of the conjecture.

An \textit{arborescence} is a directed rooted tree in which there is only one directed path from the root to any other vertex, thus all edges point away from the root. Here, we construct an infinite arborescence $G = (V, E)$, with \textit{vertex} (\textit{node}) set $V(G)$ and \textit{edge} set $E(G)$, based on the inverse mapping $g \colon \mathbb{U}_p \to \mathbb{U}$ defined by
\begin{equation} \label{eq3}
	g(x) = (2^{e(x)}x - 1)/3,\quad \forall x \in \mathbb{U}_p,
\end{equation}
where $\mathbb{U}_p = \mathbb{U} \not \equiv [0]_3$ and $e(x)$ is any positive integer such that $2^{e(x)}x - 1 \equiv [0]_3$. One easily obtains $e(x) = 2n$ if $x \equiv [1]_3$ and $e(x) = 2n-1$ if $x \equiv [2]_3$ for $n \in \mathbb{N}$, while $g(x) = \emptyset$ if $x \equiv [0]_3$. Obviously, $g$ is a \textit{one-to-many} association mapping.

Ignoring the trivial cycle, each iterate of \eqref{eq3} corresponds to a vertex in $V(G)$ and a single step of iteration for every $n \in \mathbb{N}$ represents a directed edge in $E(G)$. Collatz conjecture implies that repeatedly iterating \eqref{eq3} forms an arborescence $G$ with a root labeled $1$, excluding the trivial cycle since $G$ is inherently \textit{acyclic} (Figure \ref{Fig1}). Any finite sequence of vertices along a directed \textit{path} of at least one edge from the root in $G$ corresponds to a Collatz trajectory in reverse. Hence, we call $G$ the \textit{inverse Collatz graph}. Our aim is to show that each vertex in $G$, excluding the trivial cycle, is unique and that $V(G) = \mathbb{U}$ to validate the Collatz conjecture. 

\begin{figure} [htb]
\center{\includegraphics[scale = 0.6, width=\textwidth]{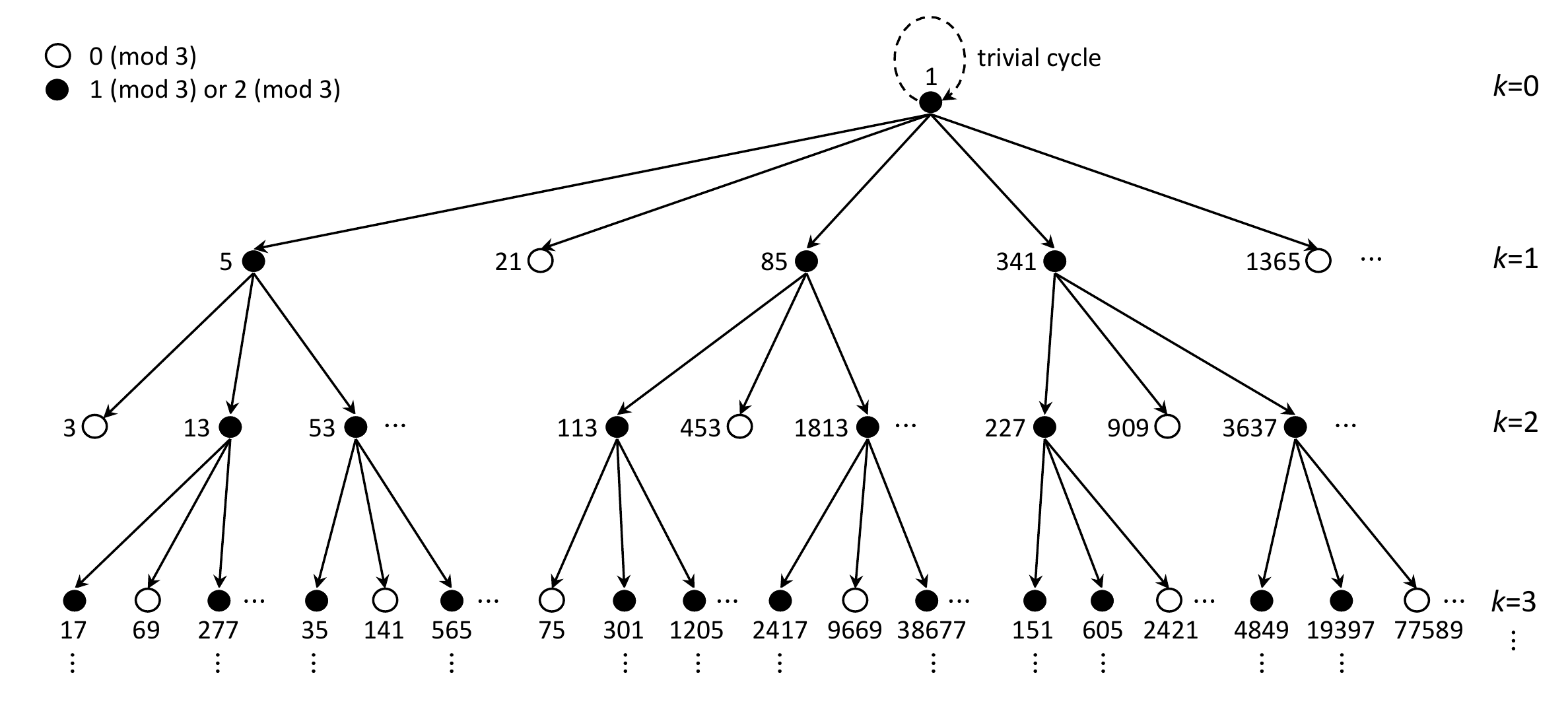}}
\caption{\label{Fig1} Structure of $G$.}
\end{figure}

\section{Constructing $G$} \label{sec2}
Given adjacent vertices $u$ and $v$, denote by $e=(u, v)$ the directed edge from $u$ to $v$, which we call \textit{parent} and \textit{child}, respectively. Children of the same parent are \textit{siblings} and form a \textit{sibling set}. Applying \eqref{eq3} leads to the \textit{transition rule} as
\begin{equation} \label{eq4}
	g_{n}(u) = v_n =
\begin{dcases}
	\emptyset, & \text{if } u \equiv [0]_3, \\
	\dfrac{2^{2n}u-1}{3}, & \text{if } u \equiv [1]_3, \\
	\dfrac{2^{2n-1}u-1}{3}, & \text{if } u \equiv [2]_3,
\end{dcases}
\end{equation}
for $n \in \mathbb{N}$. Hence, each $u \not \equiv [0]_3$ at \textit{depth} $k \in \mathbb{N}_0$ in $G$ is \textit{locally infinite} since $g_{n}(u)$ yields infinitely many children $\{v_n\}_{n \in \mathbb{N}}$. It follows that every $u \not \equiv [0]_3$ is a \textit{branching} (\textit{inner}) node, with infinitely many outgoing edges $\{e_n=(u, v_n)\}_{n \in \mathbb{N}}$, while $u \equiv [0]_3$ is a \textit{leaf} (\textit{terminal}) node, without outgoing edges. Denote by $\deg^{-}(u)$ and $\deg^{+}(u)$ the \textit{indegree} (number of incoming edges) and \textit{outdegree} (number of outgoing edges), respectively, of any $u \in V(G)$. Excluding the trivial cycle, we have for $u \in \mathbb{U}$ that
\begin{equation} \label{eq5}
	\deg^{-}(u) = 
	\begin{dcases} 
		0 \quad \text{if } u = 1, \\
		1 \quad \text{if } u > 1,
	\end{dcases}
	\quad \deg^{+}(u) =
\begin{dcases}
		0, & \text{if } u \equiv [0]_3, \\
		\aleph_0, & \text{if } u \not \equiv [0]_3,
\end{dcases}
\end{equation}
where $\aleph_0$ denotes the cardinality of $\mathbb{N}$. For $n \in \mathbb{N}$, let
\begin{equation} \label{eq6}
	z_n \coloneqq \dfrac{2^{2n}-1}{3} = \sum_{i=1}^{n} 2^{2(i-1)} = 1 + 4 + \cdots + 4^{n-1}
\end{equation}
to form the infinite strictly ascending sequence denoted by
\begin{equation} \label{eq7}
	Z = \{z_n\}_{n \in \mathbb{N}} = \{1, 5, 21, 85, 341, \ldots\}.
\end{equation}
Then for $u \in \mathbb{U}$, write $u \equiv [r]_3$ as $u=3\mu_r+r$ to restate \eqref{eq4} as
\begin{equation} \label{eq8}
	g_{n} (u) = v_n =
\begin{dcases}
	\emptyset, & \text{if } u \equiv [0]_3,\\
	z_n + 2^{2n}\mu_1, & \text{if } u \equiv [1]_3,\\
	z_n + 2^{2n-1}\mu_2, & \text{if } u \equiv [2]_3,
\end{dcases}
\end{equation}
for $n \in \mathbb{N}$. Also note the recurrence relation
\begin{equation} \label{eq9}
\begin{aligned}
	h(v_n) & = v_{n+1} =1+ 4v_n, \\
	h_{n}(v_1) & =  v_{n+1} = z_n + 4^{n}v_1.
\end{aligned}
\end{equation}
With \eqref{eq8} or \eqref{eq9}, one can specify the sequence of vertices in the sibling set of any parent $u \in \mathbb{U}_p$ in $G$.

Denote by $H_{k}(u)$ the sibling set at $k \in \mathbb{N}$ arising from $u \in \mathbb{U}_p$ at $k-1$ and let $\Pi_k$ be the union of all vertices at any $k \in \mathbb{N}_0$ in $G$. Suppose $\tau$ is the root of $G$, which is a lone vertex at $k=0$, then $\Pi_0=\{\tau\}$ is a singleton subset of $V(G)$. Applying \eqref{eq4} or \eqref{eq8} on $\tau$ yields
\begin{equation} \label{eq10}
	H_{1}(\tau) =\{g_{n}(\tau)\}_{n\in \mathbb{N}}=\{v_n\}_{n \in \mathbb{N}}=Z
\end{equation}
at $k=1$, where one should note that $e=(\tau, v_1)$ forms the trivial cycle. Since the root is a lone vertex at $k=0$, only a single sibling set with infinitely many elements exists at $k=1$. Excluding the trivial cycle, we have that
\begin{equation} \label{eq11}
	\Pi_1 \coloneqq H_{1}(\tau) \setminus \{1\} = Z \setminus \{1\},
\end{equation}
such that $\Pi_0 \cup \Pi_1 = Z$. Note that $\lvert \Pi_0 \rvert = 1$ and $\lvert \Pi_1 \rvert = \aleph_0$, where $\lvert A \rvert$ denotes the cardinality of $A$. In turn, every $v_n \not \equiv [0]_3$ in $\Pi_1$ generates an infinite sibling set at $k=2$, i.e. $\lvert H_{2}(v_n) \rvert = \aleph_0$ for every $v_n \in \mathbb{U}_p$ at $k=1$. Since $\lvert \{v_n \in \Pi_1 \colon v_n \in \mathbb{U}_p\} \rvert = \aleph_0$, it follows that a countably infinite number of siblings sets exist at $k=2$, where $\lvert \Pi_2 \rvert = \aleph_0$ given that $\aleph_0 \times \aleph_0 = \aleph_0$. In general, if $u_k \in V(G)$ is any vertex at $k \in \mathbb{N}_0$ in $G$, i.e. $u_k \in \Pi_k$, then we have that 
\begin{equation} \label{eq12}
	\Pi_{k+1} = \bigcup _{u_k \in \mathbb{U}_p} H_{k+1}(u_k).
\end{equation}
One can also restate \eqref{eq12} using a multi-index that specifically identifies each sibling set at any $k$, but such a notation becomes complicated as $k$ increases in $G$. We thus ignore any specific reference to every sibling set at higher $k$, without loss of generality. Note that
\[
\begin{aligned}
	\lvert \Pi_0 \rvert &=1 \quad \lvert \Pi_1 \rvert =\aleph_0,  \quad \lvert \Pi_2 \rvert =\aleph_0 \times \aleph_0 = \aleph_0, \quad \ldots \\
	\lvert \Pi_k \rvert &= \overbrace{\aleph_0 \times \aleph_0 \times \cdots \times \aleph_0}^{k} = \aleph_0.
\end{aligned}
\]
The above pattern holds from the root to infinitely increasing $k$, such that
\begin{equation} \label{eq13}
	\lvert \Pi_k \rvert =
\begin{dcases}
	1, & \text{if } k=0,\\
	\aleph_0, & \text{if } k>0.
\end{dcases}
\end{equation}
Hence, from the root, $G$ recursively assembles into a nested hierarchy of infinitely many sibling sets, each having a cardinality $\aleph_0$ (Figure \ref{Fig1}). As shown below, each sibling set uniquely represents an infinite sequence of odd positive integers. It is implicit from \eqref{eq13} that $G$ is both vertically infinite from the root and horizontally infinite except for the root.

\section{Patterns in $G$} \label{sec3}
Let $u \in V(G)$ and define $P \coloneqq \{u \in V(G) \colon u \in \mathbb{U}_p\}$ as the subset containing all parent nodes in $G$, such that $P \subset V(G)$. From \eqref{eq12}, it follows that $\lvert P \rvert = \aleph_0$, which implies that there exist infinitely many sibling sets in $G$. Observe that \eqref{eq8} and \eqref{eq9} permute the residue class $\pmod{3}$ of the vertices in every sibling set. The notation $\{(x, y, z)\}$ indicates that $(x, y, z)$ forms a cycle in the sequence, i.e $\{(x, y, z)\} = \{x, y, z, x, y, z, \ldots\}$ and $\{u, (x, y, z)\} = \{u, x, y, z, x, y, z, \ldots\}$. It is understood that $u \in \mathbb{U}_p$ at $k-1$ for any sibling set $H_{k}(u) = \{v_n\}_{n \in \mathbb{N}}$. We record the fact that
\begin{lemma} \label{lem1}
	The residue class $\pmod{3}$ of the vertices in $H_{k}(u)=\{v_n\}_{n \in \mathbb{N}}$ for $k>0$ takes the cycle $\{(r, r+1, r+2)\}\pmod{3}$, where $v_1 \equiv [r]_3$.
\end{lemma}

\begin{proof}
	Write $v_1 \equiv [r]_3$ as $v_1 = 3\mu_r + r$ for $r=0, 1, 2$ and apply \eqref{eq9} to yield
\begin{equation} \label{eq14}
\begin{aligned}
	v_1 \equiv [0]_3, \quad & v_2 = 3(4\mu_0)+1, & v_3 &= 3(4^{2}\mu_0+1)+2, & \ldots \\
	v_1 \equiv [1]_3, \quad & v_2 = 3(4\mu_1+1)+2, & v_3 &= 3(4^{2}\mu_1+7), & \ldots \\
	v_1 \equiv [2]_3, \quad & v_2 = 3(4\mu_2+3), & v_3 &= 3(4^{2}\mu_2+12)+1, & \ldots
\end{aligned}
\end{equation}
which correspond to the claim of the lemma.
\end{proof}

Let $w_n$ be the multiple of $z_n \equiv [r]_3$. Write $z_n = 3\mu_r+r$ for $n \in \mathbb{N}$, such that
\begin{equation} \label{eq15}
	w_n = \dfrac{z_n - r}{3} =
\begin{dcases}
	7 \sum_{i=1}^{n/3} 4^{3(i-1)}, & \text{if } n \equiv [0]_3, \\
	7 \sum_{i=1}^{(n-1)/3} 4^{3(i-1)+1}, & \text{if } n \equiv [1]_3, \\
	1 + 7 \sum_{i=1}^{(n-2)/3} 4^{3(i-1)+2}, & \text{if } n \equiv [2]_3.
\end{dcases}
\end{equation}
From \eqref{eq15}, it is easy to see that $w_n = \mu_r$ for $n \equiv [r]_3$ in $z_n$. Denote by $W$ the set of multiples corresponding to the elements in $Z$. Then
\begin{equation} \label{eq16}
	W = \{w_n\}_{n \in \mathbb{N}} = \{0, 1, 7, 28, 113, 455, \ldots\}.
\end{equation}
Similarly, let $M_{k}(u) = \{m_n\}_{n \in \mathbb{N}}$ be the set of multiples corresponding to the vertices in $H_{k}(u)=\{v_n\}_{n \in \mathbb{N}}$ at $k>0$, such that for any $v_n \equiv [r]_3$ we have that
\begin{equation} \label{eq17}
	m_n = \dfrac{v_n - r}{3}.
\end{equation}
Note that $M_{1}(\tau) = W \setminus \{0\}$ given that $H_{1}(\tau) = Z \setminus \{1\}$ in \eqref{eq11}, where $\tau$ is the root of $G$. A general pattern emerges as 
\begin{lemma} \label{lem2}
	Every $M_{k}(u) = \{m_n\}_{n \in \mathbb{N}}$ at $k>0$ is an infinite strictly ascending sequence as
\begin{equation} \label{eq18}
	M_{k}(u) =
\begin{dcases}
	\{\mu_0, w_{n-1} + 4^{n-1}\mu_0\}_{n>1}, & \textnormal{if } v_1 \equiv [0]_3, \\
	\{\mu_1, w_n + 4^{n-1}\mu_1\}_{n>1}, & \textnormal{if } v_1 \equiv [1]_3, \\
	\{\mu_2, w_{n+1} + 4^{n-1}(\mu_2-1)\}_{n>1}, & \textnormal{if } v_1 \equiv [2]_3, 
\end{dcases}
\end{equation}
where $m_1=\mu_r$ is the multiple of $v_1 \equiv [r]_3$ and $w_n$ is as above.
\end{lemma}

\begin{proof}
	One can easily check that the sequences in \eqref{eq18} results from successively iterating \eqref{eq9} on $v_1 \equiv [r]_3$ written as $v_1 = 3\mu_r + r$.
\end{proof}

\begin{lemma} \label{lem3}
	For every $H_{k}(u)=\{v_n\}_{n \in \mathbb{N}}$, write $u = 3\mu_r + r$, where $r=1, 2$. Given that $g_n(u)=v_n$, we have that
\begin{equation} \label{eq19}
\begin{aligned}
	u \equiv [1]_3 \colon \quad v_1 &=u+\mu_1, \quad v_2 = 2^{2}u + v_1, \quad v_3 = 2^{4}u + v_2, \quad \ldots \\
							v_n & = 2^{2(n-1)}u + v_{n-1} = u \sum_{i=1}^{n-1} 2^{2i} + v_1, \\						
	u \equiv [2]_3 \colon \quad v_1 &=(u+\mu_2)/2, \quad v_2 = 2u + v_1, \quad v_3 = 2^{3}u + v_2, \quad \ldots \\
							v_n &= 2^{2n-3}u + v_{n-1} = u \sum_{i=1}^{n-1} 2^{2i-1} + v_1.
\end{aligned}
\end{equation}
\end{lemma}

\begin{proof}
	From \eqref{eq4}, we obtain
\[ 
	v_1 = 
\begin{dcases}
	\dfrac{2^{2}(3\mu_1 + 1)-1}{3} = u + \mu_1, & \text{if } u \equiv [1]_3, \\
	\dfrac{2(3\mu_2 + 2)-1}{3} = \dfrac{u + \mu_2}{2}, & \text{if } u \equiv [2]_3.
\end{dcases}
\]
Then repeatedly applying \eqref{eq9} to find $v_n$ given $u \equiv [r]_3$ leads to \eqref{eq19}.
\end{proof}

Both \eqref{eq18} and  \eqref{eq19} also provide alternative ways to enumerate the sequence of vertices in $H_{k}(u)=\{v_n\}_{n \in \mathbb{N}}$ arising from any $u \in \mathbb{U}_p$. Call $v_1$ the \textit{initial vertex} in $H_{k}(u)$. From Lemma \ref{lem2}, it is clear that the sequence of initial vertices in all sibling sets at $k>1$ is monotone increasing from left to right. Recall that $\mu_1 \in \mathbb{E}$ and $\mu_0, \mu_2 \in \mathbb{U}$, thus assuring that $v_n \in \mathbb{U}$ for all $n \in \mathbb{N}$ given $u \equiv [r]_3$, where $r = 1, 2$. From \eqref{eq19}, we have that
\begin{equation} \label{eq20}
	v_{n+1}-v_n =
\begin{dcases}
	2^{2n}u, & \text{if} \ u \equiv [1]_3, \\
	2^{2n-1}u, & \text{if} \ u \equiv [2]_3,
\end{dcases}
\end{equation}
which implies that the interval between consecutive vertices in any sibling set is monotone increasing to infinity. Along with Lemma \ref{lem3}, the importance of \eqref{eq20} to our argument shall be evident subsequently. At this point, it is also relevant to note that

\begin{lemma} \label{lem4}
	Let $u_i, u_{i+1} \in \mathbb{U}_p$ be successive nodes in a sibling set at $k>0$ in $G$. If $u_i \equiv [1]_3$ for some $i \in \mathbb{N}$, then  
\begin{equation} \label{eq21}
\begin{aligned}
	g(u_i) & = v_1(u_i)= 1+ 4\mu_1 = \mu_2, \\
	g(u_{i+1}) & = v_1(u_{i+1}) = 1+2v_1(u_i) = 3+8\mu_1,
\end{aligned}
\end{equation}
where $v_1(u_i)$ and $v_1(u_{i+1})$ are the initial vertices in $H_{k+1}(u_i)$ and  $H_{k+1}(u_{i+1})$, and $\mu_1$ and $\mu_2$ are the multiples of $u_i$ and $u_{i+1}\pmod{3}$, respectively.
\end{lemma}

\begin{proof}
	From Lemma \ref{lem1}, if $u_i \equiv [1]_3$ with multiple $\mu_1$, then $u_{i+1} \equiv [2]_3$ with $\mu_2 = 1+4\mu_1$. Applying \eqref{eq8} yields $g(u_i)=v_1(u_i)=1+4\mu_1$ and $g(u_{i+1})=v_1(u_{i+1})=1+2\mu_2$. By substitution, $v_1(u_i) = \mu_2$ and $v_1(u_{i+1}) = 1+2v_1(u_i)=3+8\mu_1$, as stated in \eqref{eq21}.
\end{proof}

Given the monotone increasing sequence of infinitely many vertices in every sibling set, the least element in $\Pi_k$ is the initial vertex of the first (leftmost) sibling set at any $k \in \mathbb{N}$ in $G$. Lemma \ref{lem4} also implies that the least parent node in $\Pi_k$ always yields the least element in $\Pi_{k+1}$.

\section{Collatz is Right} \label{sec4}
Call $u \in \mathbb{U}_p > 1$ a \textit{non-trivial} parent in $G$. One easily infers from \eqref{eq8} that no such $u$ iterates to itself, i.e. $g(u) \neq u$ for any $u \in \mathbb{U}_p >1$. In addition, no $u \in \mathbb{U}_p$, including the root, can give rise to two or more equal siblings since every sibling set is an infinite strictly ascending sequence. Hence, the only scenario where a non-trivial cycle exists in $G$ rests on the possibility that distinct parents can yield equal children belonging to separate sibling sets under $g$ iteration. Assume that there exist such \textit{non-sibling} nodes $v_n, v_m \in V(G)$, where $v_n=v_m$ for $n, m \in \mathbb{N}$ and, necessarily, $n \neq m$. Without loss of generality, let $n < m$, with $v_n$ and $v_m$ being at the same or different $k$ in $G$. Suppose $u_i$ and $u_j$, where $u_i \neq u_j$, are the non-trivial parents of $v_n$ and $v_m$, respectively. We consider two cases where $u_i \equiv u_j\pmod{3}$ or $u_i \not \equiv u_j\pmod{3}$. Noting that $G$ excludes the trivial cycle, we show that
\begin{lemma} [Uniqueness] \label{lem5}
	Every vertex in $G$ is unique for all $k \in \mathbb{N}_0$.
\end{lemma}

\begin{proof}
	For $v_n=v_m$, where $u_i \not \equiv u_j\pmod{3}$, assume that $u_i \equiv [1]_3$ and $u_j \equiv [2]_3$, with multiples $\mu_1$ and $\mu_2$, respectively. If $n<m$, then $\mu_1>\mu_2$. From \eqref{eq8}, we have that
\[
	z_n + 2^{2n}\mu_1 = z_m + 2^{2m-1}\mu_2.
\]
Let $m=n+d$, where $d>0$, such that
\[
\begin{aligned}
	\sum_{i=1}^{n} 2^{2(i-1)} - \sum_{i=1}^{n+d} 2^{2(i-1)} & = 2^{2(n+d)-1}\mu_2 - 2^{2n}\mu_1, \\
	- \sum_{i=n}^{n+d} 2^{2(i-1)} & = 2^{2n} \left(2^{2d-1}\mu_2 - \mu_1\right), \\
	- \sum_{i=1}^{d} 2^{2(i-1)} & = 2^{2d-1}\mu_2 - \mu_1.
\end{aligned}
\]
Expressing for $\mu_1$, we get
\[
	\mu_1 = 2^{2d-1}\mu_2 + \sum_{i=1}^{d} 2^{2(i-1)},
\]
which is always odd and thus contradicting the fact that $\mu_1 \in \mathbb{E}$. In the case where $u_i \equiv u_j\pmod{3}$, it suffices to test only for $u_i \equiv u_j \equiv [1]_3$. For distinction, let $\mu_1$ and $\nu_1$, where $\mu_1>\nu_1$, be the multiples of $u_i $\ and $u_j\pmod{3}$, respectively, such that
\[
	z_n + 2^{2n}\mu_1 = z_m + 2^{2m}\nu_1.
\]
As above, we obtain
\[
\begin{aligned}
	\sum_{i=1}^{n} 2^{2(i-1)} - \sum_{i=1}^{n+d} 2^{2(i-1)} & = 2^{2(n+d)}\nu_1 - 2^{2n}\mu_1, \\
	- \sum_{i=n}^{n+d} 2^{2(i-1)} & = 2^{2n} \left(2^{2d}\nu_1 - \mu_1\right), \\
	- \sum_{i=1}^{d} 2^{2(i-1)} & = 2^{2d}\nu_1 - \mu_1,
\end{aligned}
\]
leading to
\[
	\mu_1 = 2^{2d}\nu_1 + \sum_{i=1}^{d} 2^{2(i-1)},
\]
which also contradicts the condition that $\mu_1 \in \mathbb{E}$. Thus, both results imply that no distinct non-trivial parents can generate equal children (non-siblings), which validates the claim of the lemma that every vertex in $G$ is unique.
\end{proof}

Lemma \ref{lem5} has the immediate consequence that
\[
	\bigcap_{k \in \mathbb{N}} \bigcap_{u \in \mathbb{U}_p} H_{k} (u) = \bigcap_{k \in \mathbb{N}_0} \Pi_k = \emptyset.
\]
A \textit{non-trivial cycle} has a sequence of vertices that does not include $1$. Lemma \ref{lem5} implies that no non-trivial cycles exist in $G$. For any $u \in V(G)$, where $u \in \mathbb{U}$, there exists $q \in \mathbb{N}$ and $r, m \in \mathbb{N}_0$, such that
\[
	u - r  = m \cdot q \quad \Longleftrightarrow \quad u \equiv r\pmod{q}.
\]
A system of $n$ linear congruences of the form
\[
	u \equiv r_i\pmod{q_i}, \quad i \in \{1,2,3,\ldots, n\}
\]
is called a \textit{covering system}, or simply a \textit{covering}, of $\mathbb{U}$ if every $u \in \mathbb{U}$ belongs to at least one of the residue classes. The concept of covering systems was first introduced by Erd{\H o}s in 1950 (\cite{Erdo}) to prove that a positive proportion of $\mathbb{U}$ cannot be expressed as the sum of a prime number and a power of $2$. The modulus $q_i$ in a covering system may be the same or distinct between residue classes. In particular, a covering of $\mathbb{U}$ is \textit{exact} if every $u$ belongs to exactly one of the residue classes. We next show that
\begin{lemma} [Completeness] \label{lem6}
	$V(G) = \mathbb{U}.$
\end{lemma}

\begin{proof}
	From \eqref{eq8}, every parent $u \equiv [1]_3$ yields $v_1 =  1 + 4\mu_1$. Since $\mu_1 \in \mathbb{E}$, we get
\[
	\mu_1 =
\begin{dcases}
	3\mathbb{E}, & \text{if } \mu_1 \equiv [0]_3,\\
	3\mathbb{U} + 1, & \text{if } \mu_1 \equiv [1]_3,\\
	3\mathbb{E} + 2, & \text{if } \mu_1 \equiv [2]_3.
\end{dcases}
\]
Similarly, if $u \equiv [2]_3$ such that $v_1 =  1 + 2\mu_2$, where $\mu_2 \in \mathbb{U}$, we obtain
\[
	\mu_2 =
\begin{dcases}
	3\mathbb{U}, & \text{if } \mu_2 \equiv [0]_3,\\
	3\mathbb{E} + 1, & \text{if } \mu_2 \equiv [1]_3,\\
	3\mathbb{U} + 2, & \text{if } \mu_2 \equiv [2]_3.
\end{dcases}
\]
By substituting for $\mathbb{E}$ and $\mathbb{U}$, noting that $u \equiv [0]_3$ is a leaf node, we derive
\begin{equation} \label{eq22}
\begin{aligned}
	u \equiv [1]_3 \colon \quad v_1 & =
\begin{dcases}
	1 + 24\mathbb{N}_0, & \text{if } \mu_1 \equiv [0]_3,\\
	17 + 24\mathbb{N}_0, & \text{if } \mu_1 \equiv [1]_3,\\
	9 + 24\mathbb{N}_0, & \text{if } \mu_1 \equiv [2]_3, \\
\end{dcases}
\\
	u \equiv [2]_3 \colon \quad v_1&  =
\begin{dcases}
	7 + 12\mathbb{N}_0, & \text{if } \mu_2 \equiv [0]_3,\\
	3 + 12\mathbb{N}_0, & \text{if } \mu_2 \equiv [1]_3,\\
	11 + 12\mathbb{N}_0, & \text{if } \mu_2 \equiv [2]_3.\\
\end{dcases}
\end{aligned}
\end{equation}
Observe that $v_1 \equiv [0]_3$, i.e. a leaf node, when $\mu_1 \equiv [2]_3$ in $u \equiv [1]_3$ and $\mu_2 \equiv [1]_3$ in $u \equiv [2]_3$; hence, $v_2 \equiv [1]_3$ is the first parent node in $H_{k}(u)$ in such cases. By applying \eqref{eq9} to \eqref{eq22}, we find that
\begin{equation} \label{eq23}
\begin{aligned}
	u \equiv [1]_3 \colon \quad H_{k}(u) & \equiv
\begin{dcases}
	\{1, (5, 21, 13)\}\pmod{24}, & \text{if } \mu_1 \equiv [0]_3,\\
	\{17, (21, 13, 5)\}\pmod{24}, & \text{if } \mu_1 \equiv [1]_3,\\
	\{9, (13, 5, 21)\}\pmod{24}, & \text{if } \mu_1 \equiv [2]_3, \\
\end{dcases}
\\
	u \equiv [2]_3 \colon \quad H_{k}(u) & \equiv
\begin{dcases}
	\{7, (5, 9, 1)\}\pmod{12}, & \text{if } \mu_2 \equiv [0]_3,\\
	\{3, (1, 5, 9)\}\pmod{12}, & \text{if } \mu_2 \equiv [1]_3,\\
	\{11, (9, 1, 5)\}\pmod{12}, & \text{if } \mu_2 \equiv [2]_3,\\
\end{dcases}
\end{aligned}
\end{equation}
where $(r_1,r_2,r_3)$ indicates a cycle of $r_i\pmod{24}$ in each sequence.

Partition $V(G)$ into subsets $V_{[1]}$ and $V_{[2]}$ defined by
\begin{equation} \label{eq24}
	V_{[1]} =  \bigcup_{k \in \mathbb{N}}  \bigcup_{u \equiv [1]_3} H_{k}(u), \quad V_{[2]} =  \bigcup_{k \in \mathbb{N}}  \bigcup_{u \equiv [2]_3} H_{k}(u),
\end{equation}
where $u \in \mathbb{U}_p$ at $k-1$ and the root $\tau$ is equal to $v_1 \in H_{1}(\tau)$. Clearly, we have that
\begin{equation} \label{eq25}
	V_{[1]} \cup V_{[2]} = V(G).
\end{equation}
Write $A \equiv \{r_i\}\pmod{q_i}$ to mean that every $a \in A$ is congruent to exactly one $r_i\pmod{q_i}$ for $i \in \{1,2,3,\ldots,n\}$. We obtain from \eqref{eq23} that
\begin{equation} \label{eq26}
	V_{[1]} \equiv \{1, 5, 9\}\pmod{12}, \quad V_{[2 ]}\equiv \{1, 3, 5, 7, 9, 11\}\pmod{12},
\end{equation}
which reveals that no element in $V_{[1]}$ is congruent to $\{3,7,11\}\pmod{12}$. This is easy to confirm since from \eqref{eq1} we have that
\[
\begin{aligned}
	f(3+12\mathbb{N}_0) &= 5+18\mathbb{N}_0=2+3(1+6\mathbb{N}_0),\\
	f(7+12\mathbb{N}_0) &= 11+18\mathbb{N}_0=2+3(3+6\mathbb{N}_0),\\
	f(11+12\mathbb{N}_0) &= 17+18\mathbb{N}_0=2+3(5+6\mathbb{N}_0),
\end{aligned}
\]
which are all congruent to $[2]_3$, but $V_{[1]}$ only consists of vertices arising from $f(x)=u \equiv [1]_3$.

Let $\mathcal{V}_1 = \{v_1 \in V_{[1]}\}$ and $\mathcal{V}_2 = \{v_1 \in V_{[2]}\}$, which means that $\mathcal{V}_r$ contains all initial vertices in $V_{[r]}$ for $r=1, 2$, such that $\mathcal{V}_1 \subset V_{[1]}$ and $\mathcal{V}_2 \subset V_{[2]}$. It is trivial that
\begin{equation} \label{eq27}
	\mathcal{V}_1 \cap \mathcal{V}_2 = \emptyset,
\end{equation}
since $g(u) \neq g(v)$ for any $u \equiv [1]_3$ and $v \equiv [2]_3$. To show this, let $\mu_1$ and $\mu_2$ be the multiples of $u$ and $v$, respectively. Then one obtains
\begin{equation} \label{eq28}
\begin{aligned}
	1+4\mu_1 &= 1+2\mu_2,\\
	2\mu_1 & = \mu_2
\end{aligned}
\end{equation}
which is false since $\mu_1 \in \mathbb{E}$ and $\mu_2 \in \mathbb{U}$. In fact, one easily finds from \eqref{eq8} that $\mathcal{V}_1 = \{1+8\mathbb{N}_0\}$ and $\mathcal{V}_2 = \{3+4\mathbb{N}_0\}$, which are mutually exclusive. Hence, \eqref{eq27} also implies Lemma \ref{lem5} since every sibling set is an infinite strictly ascending sequence that is fully determined by its initial element. It follows that $V_{[1]}$ and $V_{[2]}$ are mutually disjoint, i.e.
\begin{equation} \label{eq29}
	V_{[1]} \cap V_{[2]} = \emptyset.
\end{equation}
Since $\{1, 3, 5, 7, 9, 11\}\pmod{12}$ exactly covers $\mathbb{U}$ and $V_{[1]} \cup V_{[2]} \equiv \{1, 3, 5, 7, 9, 11\}\pmod{12}$, with each vertex in $G$ being unique, it follows that $V(G) = \mathbb{U}$ and the proof is complete.
\end{proof}

It is obvious that $G$ is \textit{weakly connected} since if it is treated as an undirected graph, then every pair of distinct vertices $u$ and $v$ can be joined by an undirected path. For any $u \in P$, where $P$ is the subset containing all parent nodes in $G$ such that $P \subset V(G)$, notice that $v_1>u$ if $u \equiv [1]_3$ and $v_1<u$ if $u \equiv [2]_3$. We say that $e=(u, v_1)$ is \textit{ascending} if $u \equiv [1]_3$ or \textit{descending} if $u \equiv [2]_3$. Any path $p(u_0,u_k) = \{u_0, u_1, u_2, \ldots, v_k\}$ is strictly ascending or descending if $u_i \in V_{[1]}$ or $u_i \in \mathcal{V}_2$, respectively, for all $i=0,1,2,\ldots, k-1$ in $G$. The \textit{end vertex} $u_k$ in $p(u_0, u_k)$ can be any element in the sibling set arising from $u_{k-1}$, such that $u_i \in \mathbb{U}_p$ for all $i=1,2,3,\ldots, k-1$ and $u_k \in \mathbb{U}$. If $u_k$ is a leaf node, i.e. $u_k \equiv [0]_3$, then $p(u_0,u_k)$ cannot extend to any other vertices beyond $k$, in which case we say that $u_k$ is a \textit{dead end} and $p(u_0, u_k)$ is \textit{truncated} at $u_k$. This leads to the fact that every path from the root in $G$ eventually terminates in a dead end. Hence, there can only be at most one vertex $\equiv [0]_3$ in any $p(u_0, u_k)$ and if it exists, it must be $u_k$.

To reinforce the proof of Lemma \ref{lem6}, we further demonstrate how $V(G)$ exactly covers $\mathbb{U}$. Consider a number line representing only the odd positive integers. In every $H_{k}(u)=\{v_n\}_{n \in \mathbb{N}}$ for $u \equiv \mathbb{U}_p$ and $k \in \mathbb{N}$, both \eqref{eq8} and \eqref{eq19} specify that $v_1=u+\mu_1=1+4\mu_1$ if $u \equiv [1]_3$ and $v_1=(u+\mu_2)/2=1+2\mu_2$ if $u \equiv [2]_3$, such that $u < v_1$ if $u \equiv [1]_3$ and $u > v_1$ if $u \equiv [2]_3$ as claimed above. The root $\tau$ gives rise to $H_{1}(\tau) = Z = \{1, 5, 21, 85, \ldots\}$, where $e=(\tau, v_1)$ forms the trivial cycle that is ignored in $G$. Mark $H_{1}(\tau)$ as a sequence on the odd number line at $k=1$. Also apply $g$ on every $v_n \in \mathbb{U}_p > 1$ in $H_{1}(u_0)$ to generate infinitely many sibling sets, whose vertices are also marked on the number line at $k=2$. Repeat the same process for increasing $k$ to infinity, as depicted in Figure \ref{Fig2} with only the first few vertices shown for the initial sibling sets at $k=1$ to $7$.
	
In this context, a \textit{gap} refers to the number of odd positive integers between two successive vertices in a sibling set. In $H_{1}(\tau)$, there exist exactly $2^{2n-1}-1$ odd integers between $v_n$ and $v_{n+1}$, not inclusive, for $n \in \mathbb{N}$. Hence, Lemma \ref{lem7} is true if the subsequent $g$ iterations uniquely generate all odd integers between vertices in $H_{1}(\tau)$. Ignoring the trivial cycle, the first sibling set $H_{2}(5)=\{3, 13, 53, \ldots\}$ at $k=2$ arises from $v_2=5$ in $H_{1}(\tau)$. Notice that each vertex in $H_{2}(5)$ occupies the midpoint between successive vertices in $H_{1}(\tau)$ (Figure \ref{Fig2}). Since the initial vertex completely determines the spacing between successive vertices in a sibling set, it is crucial to examine how the subsequent initial vertices are generated to fill the gaps in $H_{1}(\tau)$. Observe that $v_1$ at $k>1$ falls to the right or left of its parent $u$, depending on whether $u \equiv [1]_3$ or $[2]_3$, respectively, since
\begin{equation} \label{eq30}
	v_1 = u + \mu_1 \quad \text{if } u \equiv [1]_3, \qquad v_1 = u - (\mu_2 + 1) \quad \text{if } u \equiv [2]_3.
\end{equation}
With $\mu_1 \in \mathbb{E}$ and $\mu_2 \in \mathbb{U}$, it is assured that $v_1$ is always odd in \eqref{eq30}. Equating the expressions in \eqref{eq30} leads to \eqref{eq28}, thus implying that no two initial vertices in distinct sibling sets are equal at any $k>1$ in $G$.

\begin{figure} [htb]
\center{\includegraphics[scale = 0.6, width=\textwidth]{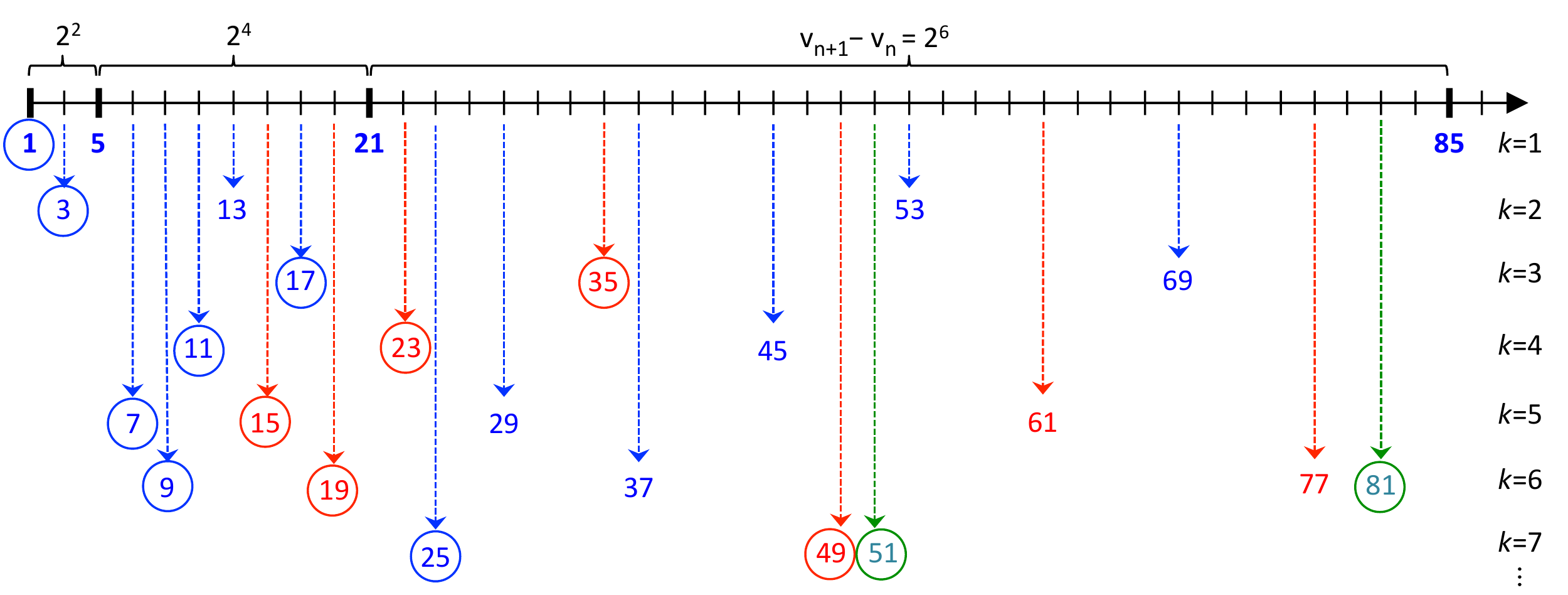}}
\caption{\label{Fig2} Gap-filling process by $g$ iteration. Numbers of the same color indicate the first few vertices (siblings) in the same sibling set, with the encircled number as the initial vertex, at each depth $k$ in $G$.}
\end{figure}

By applying \eqref{eq30} on any $u \not \equiv [0]_3$ at $k$ with respect to the odd number line to yield $v_1 \in \mathbb{U}$ at $k+1$, one can generate the sequence of vertices in $H_{k+1}(u)$ by any method given above. Since the gap between $v_n$ and $v_{n+1}$ strictly increases with $n$ in every sibling set at $k \in \mathbb{N}$, repeatedly iterating \eqref{eq30} eventually produces all odd integers to fill such gaps. The gaps tend to be fully filled in progression from left to right in $H_{1}(\tau)$, thereby completely covering the odd number line as $k$ goes to infinity. Since $v_1$ is a leaf node in $H_{2}(5)$, further iterations on $v_2=13$ produce the vertices $\{13, 17, 11, 7, 9\}$ at $k=2$ to $6$, which all lie between $5$ and $21$ in $H_{1}(\tau)$ (Figure \ref{Fig2}). For the other odd integers in the same gap, $15$ results from successive iterations on $v_3=53$ to yield $\{53, 35, 23, 15\}$ at $k=2$ to $5$, while $19$ is the initial vertex arising from $v_2=29$ at $k=5$. As $9$ and $15$ are leaf nodes, both paths $p(\tau, 9)$ and $p(\tau, 15)$ thus terminate in dead ends. The siblings sets associated with these initial vertices, e.g. as determined by \eqref{eq9}, contain infinitely many $u \in \mathbb{U}_p$ on which \eqref{eq30} also applies recursively. In fact, further recursions on $u \in \mathbb{U}_p$ up to infinity will generate all distinct odd numbers to fill the gaps in $H_{1}(\tau)$, thus eventually covering the entire odd number line. In short, this \textit{gap-filling} process by $g$ iteration (Figure \ref{Fig2}) generates $\Pi_k$, with its elements ordered in monotone increasing sequence on the odd number line, for all $k \in \mathbb{N}$. Therefore, since every vertex in $V(G)$ is unique by Lemma \ref{lem5} and $V(G)=\mathbb{U}$ by Lemma \ref{lem6}, we deduce from the gap-filling process that $V(G)$ is a complete and exact covering of $\mathbb{U}$.

The gap-filling process assures that no odd number is not covered by $V(G)$, as claimed by Lemma \ref{lem6}. Assume the contrary that there exists $u_k \in \mathbb{U}$, where $u_k \not \in V(G)$ such that $V(G) \neq \mathbb{U}$. It suffices to consider only $u_k \equiv [0]_3$, since, just the same, $u_k$ yields $H_{k+1}(u_k)$, with infinitely many leaf nodes to which every directed path eventually terminates, if $u_k \not \equiv [0]_3$. This implies that it is impossible to have any finite subset of $\mathbb{U}$ that is not in $V(G)$. It is always true that $u_k = z_n + 2^{2n-1}c$ or $u_k = z_n + 2^{2n}c$ for $c \in \mathbb{N}_0$, thus implying that there exists $u_{k-1} \in \mathbb{U}_p$ such that $g_{n}(u_{k-1})=u_k$ for $n \in \mathbb{N}$. It follows that $u_{k-1}$ is the parent of $u_k$, where $c=\mu_r$ given that $u_{k-1}=3\mu_r + r$ for $u_{k-1} \equiv [r]_3$. Suppose $u_{k-1}$ is also not in $G$, then there exists $u_{k-2} \not \equiv [0]_3$ that is the parent of $u_{k-1}$. In general, consider $\{u_i \in \mathbb{U}_p \colon g(u_i) = u_{i+1}, \forall i=k-2, k-3, \ldots, 2, 1, 0\}$ for some $k \in \mathbb{N}$, where $v_i \not \in V(G)$ for all $i$. Then we state the fact that $p(u_0, u_{k-1})$ is a path where $\deg^{-}(u_0) = 0$ and $\deg^{-}(u_i) = 1$ for $i=1,2,3, \ldots, k-1$, whereas $\deg^{+}(u_i) = \aleph_0$ for all $i=0, 1,2, \ldots, k-1$. It is impossible that $u_0>1$ because if such is indeed the case, then there must exist another vertex that is parent to $u_0 \not \in V(G)$, which is a contradiction. Hence, the only likelihood is that $u_0 =1$, which corresponds to the root $\tau$ of $G$, such that $p(u_0, u_{k-1})$ is connected to $G$. Since every parent in $G$ gives rise to one and only one sibling set $H_{k}(u_{k-1})$, with infinitely many elements, then $u_k \in H_{k}(u_{k-1})$. Therefore, $u_k$ is in $G$ and there exists a directed path $p(\tau, u_k)$ from the root $\tau$ to $u_k$, thus contradicting our assumption that $u_k$ is disjoint from $G$.

Call $u$ a \textit{non-initial} vertex if it is not the initial vertex in a sibling set. Any path consisting entirely of non-initial vertices in $G$ is always strictly ascending. If a path is a mix of ascending and descending edges, then it is called a \textit{hailstone path}. We refer to any proper subset of a path $p$ as \textit{sub-path} or \textit{segment} $s$, such that $s \subseteq p$. Obviously, $p$ is a subset of itself. It is a trivial fact that infinitely many vertices in a sibling set, say, $H_{k}(u)$ share the same sub-path $s(u_0, u)$ from the root $u_0$ to the parent node $u \not \equiv [0]_3$ at $k-1$ in $G$. This implies that there exist infinitely many odd natural numbers with convergent Collatz trajectories of equal length $k \in \mathbb{N}$ under  $f$ iteration given in \eqref{eq1}. In fact, we now claim that
\begin{lemma} [Convergence] \label{lem7}
	Every odd natural number has a convergent trajectory under $f$ iteration.
\end{lemma}

\begin{proof}
	As $G$ is an arborescence, there is always a directed path from the root $\tau$ to any other vertex $u \in V(G)$. Since $V(G) = \mathbb{U}$ by Lemma \ref{lem6} and every path $p(\tau, u)$ is a convergent Collatz trajectory in reverse, then the claim of the lemma follows.
\end{proof}

A trivial corollary to Lemma \ref{lem7} is the fact that the Collatz trajectory of any odd natural number, as defined by \eqref{eq2}, can never diverge. Hence, with Lemmas \ref{lem5}-\ref{lem7}, we have shown why the Collatz conjecture is true.

\section*{Acknowledgments}
Critical remarks from an anonymous reviewer of an earlier version of the manuscript substantially improved the approach presented in this paper. The support from KAUST Core Labs is deeply appreciated.

\end{document}